\documentclass{amsart}

\usepackage{amssymb}
\usepackage{amsmath}
\usepackage{amsthm}
\usepackage{verbatim}
\edef\restoreparindent{\parindent=\the\parindent\relax}
\usepackage[parfill]{parskip}
\usepackage{diagbox}
\usepackage{thmtools}
\usepackage{enumitem}
\usepackage{hyperref}
\usepackage{thm-restate}

\usepackage[backend=bibtex,maxbibnames=99]{biblatex}
\def\indep{\mathrel{\raise0.2ex\hbox{\ooalign{\hidewidth$\vert$\hidewidth\cr\raise-0.9ex\hbox{$\smile$}}}}}

\newcommand{\Piin}{\Pi^{\mathrm{in}}}
\newcommand{\Siin}{\Sigma^{\mathrm{in}}}

\newcommand{\Ltype}{\mc L_{\operatorname{type}}}
\newcommand{\Loin}{\mc L_1^{\mathrm{in}}}
\newcommand{\Linf}{\mc L_{\omega_1\omega}}

\newcommand{\ctb}{{\aleph_0}}

\let\epsilon\varepsilon
\newcommand{\vp}{\varphi}
\let\phi\varphi

\newcommand{\mc}{\mathcal}
\newcommand{\mf}{\mathfrak}
\newcommand{\sbeq}{\subseteq}
\newcommand{\speq}{\supseteq}
\newcommand{\sat}{\vDash}
\newcommand{\bsl}{\backslash}

\newcommand{\set}[1]{\left\{#1\right\}}
\newcommand{\inv}{^{-1}}
\newcommand{\ang}[1]{\left\langle#1\right\rangle}

\mathchardef\mhyph="2D

\let\bar\overline 
\let\0\varnothing 
\let\emptyset\varnothing 
\let\hat\widehat

\DeclareMathOperator{\Th}{Th}
\DeclareMathOperator{\tp}{tp}
\DeclareMathOperator{\dom}{dom}
\DeclareMathOperator*{\bigdoublewedge}{\bigwedge\mkern-15mu\bigwedge}

\theoremstyle{plain}
\newtheorem{thm}{Theorem}[section]
\newtheorem{cor}[thm]{Corollary}
\newtheorem{prop}[thm]{Proposition}
\newtheorem{lem}[thm]{Lemma}
\newtheorem{conj}[thm]{Conjecture}

\theoremstyle{definition}
\newtheorem{defn}[thm]{Definition}
\newtheorem{notn}[thm]{Notation}
\newtheorem{rmk}[thm]{Remark}
\newtheorem{quest}[thm]{Question}

\theoremstyle{remark}

\title{The Type $\omega$-Vaught's Conjecture for $\omega$-stable Theories}
\author{Hongyu Zhu}

\address{University of Wisconsin-Madison,
Department of Mathematics,
480 Lincoln Drive, Madison, WI, 53706, USA}
\email{\href{mailto:hongyu@math.wisc.edu}{hongyu@math.wisc.edu}}
\urladdr{\url{https://people.math.wisc.edu/\string~hzhu322/}}

\begin{document}
\begin{abstract}
  We propose a new variant of Vaught's conjecture, based on Martin's (model-theoretic) conjecture and the $\omega$-Vaught's conjecture of Gonzalez and Montalb\'an. For $\omega$-stable theories, we prove the conjecture by analyzing in detail Bouscaren's and Shelah-Harrington-Makkai's proofs of Martin's conjecture and Vaught's conjecture, respectively, for $\omega$-stable theories. In particular, when there are countably many countable models, we improve the required number of quantifiers for a Scott sentence from $\omega+\omega$ to $\omega+5$, and we obtain a similar bound of $\omega+6$ for the uncountable case.
\end{abstract}
\maketitle

\section{Introduction}
One of the longest standing open questions in model theory is Vaught's conjecture \cite{vaught1961denumerable}, which states that the number of countable models of a first-order theory must be either continuum or at most countable (under the negation of the continuum hypothesis). Various approaches have been taken to address the question, with methods ranging from model theory \cite{BuechlerVC}, descriptive set theory \cite{topVC}, to computability theory \cite{MontalbanVC}.

One approach, Martin's model-theoretic conjecture (\autoref{martin}), attempts to describe the countable models when there are fewer than $2^\ctb$ many of them. More precisely, the description is to be a countable theory of first-order sentences in the complete types of $T$ (i.e. sentences in $\mc L_1(T)$; see \autoref{L1}) which determines the model up to isomorphism, and Martin's conjecture asserts the existence of such descriptions when the number of countable models is fewer than continuum.

A more recent approach is the $\omega$-Vaught's conjecture of Gonzalez and Montalb\'an \cite{wVC} (see also \autoref{wVC}). Similar to Martin's conjecture, the $\omega$-Vaught's conjecture asserts the existence of descriptions for countable models as well. However, both the descriptions and the models which are involved are different. First, the description is now an infinitary sentence (i.e. in $\Linf$). Despite the fact that all countable models have a sentence in $\Linf$ that characterizes it up to isomorphism (known as a Scott sentence) \cite{Scott1965}, further restrictions on the complexity of the infinitary sentence are placed. Second, the $\omega$-Vaught's conjecture requires certain behaviors on the models even when $T$ has continuum many models; Specifically, if $T$ is $\Piin_\alpha$, then at a complexity of $\Piin_{\alpha+n}$ with $n<\omega$, one should already be able to distinguish continuum many models.

In combination of both conjectures, we propose a new conjecture, the Type $\omega$-Vaught's conjecture (\autoref{twVC}). The language $\mc L_1^{\text{in}}(T)$ (see \autoref{L1}) is slightly more than just $\mc L_1(T)$, though complexity-wise it is no more than that of $\mc L_1(T)$: We allow infinitary conjunctions and disjunctions for finitely many levels. In addition, we require behaviors from the models whether $T$ has at most $\ctb$ models or $2^\ctb$ models: When it has at most $\ctb$, there needs to be a Scott sentence in the aforementioned $\mc L_1^{\text{in}}(T)$; And when continuum many, the same language should be able to distinguish continuum many of them. The main result of this paper is to establish the Type $\omega$-Vaught's conjecture for $\omega$-stable theories, using Shelah-Harrington-Makkai's proof of Vaught's conjecture for $\omega$-stable theories \cite{SHM} and Bouscaren's proof of Martin's conjecture for $\omega$-stable theories \cite{Bouscaren}. In particular, when there are fewer than $2^\ctb$ many countable models, we improve the number of quantifiers needed from $\omega+\omega$ (as in Bouscaren's proof) to $\omega+5$, and we obtain a similar bound of $\omega+6$ for the uncountable case.

We note that there have been other similar approaches to Vaught's conjecture. In particular, the strong Martin's conjecture proposed by Wagner \cite{wagner_smc} requires, in addition to Martin's conjecture, that the theory should have $2^\ctb$ completions in $\mc L_1(T)$ when it has $2^\ctb$ countable models. Itai \cite{Itai_1991} did a similar analysis of the strong Martin's conjecture for $\omega$-stable theories, though the conjecture in its full generality was shown to be false by Gao \cite{gao_smc}.

This article contains four sections, starting with the current introductory one. The second section defines formally the three aforementioned conjectures, namely Martin's conjecture, the $\omega$-Vaught's conjecture, and the Type $\omega$-Vaught's conjecture. We present a full proof of our main result, the Type $\omega$-Vaught's conjecture for $\omega$-stable theories, in the third section, and conclude with a series of open questions in the fourth and last section.

\section{Background}
Throughout the remainder of the article, we work in a fixed countable language $\mc L$.

\subsection{Martin's conjecture}
First, recall Martin's conjecture: Assume $T$ is complete first-order $\mc L$-theory.
\begin{defn}\label{L1}
Let $\Ltype(T)=\mc L\cup\set{\bigdoublewedge_{\vp(\bar x)\in p(\bar x)}\vp(\bar x)\mid p(\bar x)\in S_n(T),n\in\omega}$ be the language obtained from $\mc L$ by adding all the complete types over $T$ as additional predicates.

Let $\mc L_1(T)$ be the first-order closure of $\Ltype(T)$, i.e. the smallest fragment of $\Linf$ containing $\Ltype(T)$ (closed under $\wedge, \neg$, and universal quantification).
\end{defn}

\begin{conj}[Martin]\label{martin}
Suppose $T$ has fewer than $2^\ctb$ countable models. Then the $\mc L_1(T)$-theory of any $\mc M\sat T$ is $\ctb$-categorical.
\end{conj}

\begin{rmk}
  To obtain a Scott sentence we need to take an infinite conjunction over the $\mc L_1(T)$-theory, which would then be $\Piin_{\omega+\omega}$ over the base theory.
\end{rmk}

\subsection{$\omega$-Vaught's conjecture}
Next, we discuss the $\omega$-Vaught's conjecture of Gonzalez and Montalb\'an \cite{wVC} and generalize to arbitrary limit ordinals $\alpha$: Assume instead  $T$ is an $\Linf$-sentence.
\begin{defn}
The \emph{Vaught ordinal} of $T$, $vo(T)$, is the least $\beta$ such that:
\begin{itemize}
  \item If $T$ has countably many countable models, then all of them have Scott rank $<\beta$.
  \item If $T$ has uncountably many countable models, then there are uncountably many of them which are not $\Piin_\beta$-equivalent to each other.
\end{itemize}
\end{defn}
\begin{rmk}
The conjecture we are about to state bounds the Vaught ordinal above by some limit ordinal. As a result, it mostly does not matter which notion of Scott rank we use, since most variants differ only by a finite amount. For example, we can replace ``have Scott rank $<\beta$'' above by ``have a $\Piin_\beta$ Scott sentence,'' or use the robuster notion of Scott rank of Montalb\'an \cite{robusterSR}.
\end{rmk}

\begin{conj}[$\alpha$-Vaught's conjecture]\label{wVC}
For a limit ordinal $\alpha$, say $T$ satisfies the \emph{$\alpha$-Vaught's conjecture} if for every $\beta<\omega_1$ such that $T$ is equivalent to a $\Piin_\beta$-sentence, we have $vo(T)\le\beta+\alpha$.
\end{conj}
\begin{rmk}
This version differs slightly from the original formulation of $\omega$-Vaught's conjecture, which requires both $T$ and all (infinitary) sentences that imply $T$ to satisfy the bound on the Vaught ordinal. As we work with only first-order theories, where the only consistent theory that can imply a complete theory $T$ is $T$ itself, this statement should be sensible to make.
\end{rmk}

\subsection{Type $\omega$-Vaught's conjecture}
Now we propose a new conjecture, the Type $\omega$-Vaught's conjecture, inspired by both conjectures above. Let $T$ be a complete first-order theory again.
\begin{defn}
Let $\mc L_1^{\text{in}}(T)$ be set of $\Piin_{<\omega}$-sentences over $\Ltype(T)$ (equivalently over $\mc L_1(T)$), where $\Piin_{<\omega}=\bigcup_{n<\omega}\Piin_{n}$.
\end{defn}
\begin{rmk}
\begin{itemize}
  \item $\mc L_1(T)\sbeq\mc L_1^{\text{in}}(T)$.
  \item $\mc L_1(T)$ and $\mc L_1^{\text{in}}(T)$ have the same complexity: If all types over $T$ are $\forall_n$-axiomatizable for some $n$, then $\mc L_1(T)$ and $\mc L_1^{\text{in}}(T)$ are both subsets of $\Piin_{<\omega}$; otherwise, in general both will contain strictly $\Piin_\omega$-sentences, but are still subsets of $\Piin_{<\omega+\omega}$.
\end{itemize}
\end{rmk}

\begin{conj}[Type $\omega$-Vaught's conjecture]\label{twVC}
For any complete first-order theory $T$:
\begin{itemize}
  \item If $T$ has countably many countable models, then any countable $\mc M\sat T$ has a Scott sentence in $\Loin(T)$.
  \item If $T$ has uncountably many countable models, then there are $2^\ctb$ countable models of $T$ which are pairwise not $\mc L_1^{\mathrm{in}}(T)$-equivalent.
\end{itemize}
\end{conj}

\begin{rmk}\label{wVCcomparison}
The Type $\omega$-Vaught's conjecture for $T$ implies the $(\omega+\omega)$-Vaught's conjecture for $T$. If in addition that either every type over $T$ is \emph{boundedly axiomatizable} (i.e. $\forall_n$-axiomatizable for some $n\in\omega$) or $T$ is not boundedly axiomatizable, then it actually implies the $\omega$-Vaught's conjecture for $T$.
\end{rmk}

\begin{notn}
  When $T$ is obvious from context, we will abbreviate $\mc L_1(T)$ as just $\mc L_1$, and similarly for $\Loin$ and $\Ltype$.
\end{notn}

Our main result is stated as follows.

\begin{restatable*}{thm}{main}\label{main}
The Type $\omega$-Vaught's conjecture holds for any complete $\omega$-stable theory $T$. In fact, we obtain uniform bounds on the quantifier complexity:
\begin{itemize}
  \item If $T$ has countably many countable models, then any countable $\mc M\sat T$ has a $\operatorname{d-\Siin_4}(\Ltype)$ Scott sentence.
  \item If $T$ has uncountably many countable models, then there are $2^\ctb$ countable models of $T$ which are pairwise not $\Piin_6(\Ltype)$-equivalent.
\end{itemize}
\end{restatable*}

The uniform finite bounds in the theorem gives a more precise result in view of the $\alpha$-Vaught's conjecture:

\begin{cor}\label{wVCbound}
  The $(\omega+6)$-Vaught's conjecture holds for all complete $\omega$-stable theories.
\end{cor}

The analysis relies heavily on the proofs of Vaught's conjecture for $\omega$-stable theories by Shelah-Harrington-Makkai \cite{SHM} and of Martin's conjecture for $\omega$-stable theories by Bouscaren \cite{Bouscaren}.

\section{The Analysis}
This section is devoted to proving the main result, \autoref{main}. We will start in the first subsection by proving definability results used throughout the entire section. This gives us enough power to prove the result for when $T$ has countably many countable models in the second subsection. For the uncountable case, we start by justifying working over a fixed finite tuple of parameters in the third subsection, and then discuss the four subcases (as in \autoref{ctbconds}) in the fourth subsection. We exhibit distinguishing formulas for each of the four cases in the remaining four subsections (which together prove the uncountable case).

From now on, let $T$ be a fixed complete $\omega$-stable theory.

\subsection{Expressing non-forking independence}
We first show that the non-forking independence relation is expressible in $\Loin$. The main tool we will use is the following proposition. Recall that if $A\equiv B$ via some elementary map $f:A\to B$ and $p$ is a type over $A$, then $p_B$ denotes the type $f(p)$ (which is a type over $B$). In particular, if $a$ realizes $p$, then for any $b$, we have $b$ realizes $p_B$ if and only if $\tp(Aa)=\tp(Bb)$.

\begin{prop}\label{indep}
For any type $p\in S_n(T)$. There is a quantifier-free $\mc L_{type}$-formula $Indep_{p}(\bar x,\bar y,\bar z)$ (technically, a different formula for each $n$ and each arity of $\bar z$) such that for any $\mc M\sat T$ and $n$-tuples $\bar a,\bar b, \bar c$ in $\mc M$ with matching arities:

$\mc M\sat Indep_{p}(\bar a,\bar b,\bar c)\iff(\tp(\bar {ab})=p)\wedge(\bar c\indep_{\bar a}\bar b)$.
\end{prop}
\begin{proof}
  The theory is $\omega$-stable, thus has finite multiplicity, so in particular there are only finitely many nonforking extensions of any type to any (finite) set. Thus, independence can be phrased as $\bar{zy}$ realizing any of the finitely many nonforking extensions of $\tp(\bar y/\bar x)$ (which can be found from $p$), which is just a disjunction of finitely many predicates and thus quantifier-free expressible in $\mc L_{type}$.
\end{proof}

In addition, dimension is also expressible in $\Loin$.

\begin{prop}\label{dim}
Fix $l<\omega$. Let $r\in S_n(T)$. Let $p$ be a type over a finite set (of size $n$) realizing $r$. Let $s\in\omega+1$. Then there is an $\Loin(T)$ formula $\psi_{r,p,s}(\bar x)$ such that for any $\mc M\sat T$ and tuple $\bar m$ in $\mc M$ matching the arity of $r$:

$\mc M\sat \psi_{r,p,s}(\bar m)\iff\tp(\bar m)=r\wedge\dim(p_{\bar m},\mc M)=s$.

The complexity of $\psi_{r,p,s}$ is as follows:
    \begin{itemize}
      \item In general, $\psi_{r,p,s}$ can be taken to be $\Piin_2$ over $\mc L_{type}$.
      \item If $s<\omega$, then $\psi_{r,p,s}$ can be taken to be $d\mhyph\exists_1$ over $\mc L_{type}$.
    \end{itemize}
\end{prop}

\begin{proof}
  Note that the first claim follows from the second: If $s=\omega$, then we can simply use the formula that says $r(\bar x)\wedge(\bigdoublewedge_{s\in\omega}\neg\psi_{r,p,s})$, which is now $\Piin_2$ over $\mc L_{type}$ (and thus in $\Loin(T)$).

  For the second claim, assume $s\in\omega$. The construction follows immediately from the ``first remark'' in \cite[Proposition 4]{Bouscaren}.\footnote{It should be noted that while in that proposition Bouscaren claims $\psi_{r,p,s}$ can be written in $\mc L_1(T)$ even for $s=\omega$, that uses the assumption of $T$ having countably many models. However, the remark for $s\in\omega$ does not use this assumption.}
  To be more explicit: The first part that says $\bar x$ realizes $r$ is quantifier-free because $r(\bar x)$ is in the language $\mc L_{type}$. The second part just says there are $s$ (but not $s+1$) realizations of $p_{\bar m}$ which are independent. Saying $\bar y$ realizes $p_{\bar x}$ is expressible by $q(\bar y,\bar x)$ where $q$ is the type of $\bar aA$ with $\bar a$ any realization of $p$ and $A$ the base of $p$. (Note again that $q$ is in the language so this is quantifier-free.)
  At last it remains to express a tuple of variable (realizing the same type $p$) being independent, but we can just apply \autoref{indep}. Hence $\psi_{r,p,\omega}$ is $d\mhyph\exists_1$ over $\mc L_{type}$. This establishes the second claim.
\end{proof}
\begin{rmk}
It is crucial in the above propositions that the formulas depend on the types of the free variables. This allows us to know (non-uniformly) the finitely many nonforking types, reducing the complexity of the independence relation down. The story would be different if we were to give formulas uniformly.
\end{rmk}

\subsection{The countable case}

First, consider the case when $T$ has countably many countable models.

\begin{thm}
  Every countable model of $T$ has a Scott sentence that is $\Siin_5$ over $\Ltype$.
\end{thm}
\begin{proof}
   We use essentially the same Scott sentence as in \cite[Proposition 5]{Bouscaren}, but improve on the number of quantifiers required.. First of all, everything that follows depends on a fixed tuple $a^0$ of a given isolated type, which takes an (outermost) existential quantifier to pin down. The Scott sentence has four parts, defined using a given fixed tuple $d^0$ and finite sets of types $P=\set{p^i\mid i<l}$, $Q$, and $R$. Before we give the sentence, though, it is necessary to introduce another piece of notation from \cite{Bouscaren}:
   \begin{defn}
     $D(\bar m)=\bar u$ means the dimension of $p^i_{\bar m}$ is $s_i$ for all~$i$.
   \end{defn}
   The following result improves the na\"ive complexity of the formula $D(\bar m)=\bar s$, which is a priori $\Piin_2$ over $\Ltype$ by \autoref{dim}.
   \begin{lem}\label{Ddimension}
     If $\tp(\bar m)=r$ is fixed, then $D(\bar m)=\bar s$ can be expressed as a formula (with $\bar m$ as the variables) that is $\forall_2$ over $\Ltype$ (in particular in $\mc L_1(T)$).
   \end{lem}
   \begin{proof}
     We use a related but slightly different method as in the proof of \cite[Proposition 4]{Bouscaren}. Given tuples $\bar u,\bar t\in(\omega+1)^l$, say $\bar u\sim_r\bar t$ if $D(\bar m)=\bar u, D(\bar n)=\bar t$ for some tuples $\bar m\not\indep\bar n$. \cite[Proposition 2 and Lemma 6]{Bouscaren} shows that $\bar u\sim_r\bar t$ implies $\bar u$ and $\bar t$ have the same number of infinite coordinates, and \cite[Proposition 3]{Bouscaren} shows that there are only finitely many $\sim_r$-equivalence classes.

     Define $\bar u\lesssim\bar t$ as the following: for all $i$ such that $(\bar t)_i<\omega$, $(\bar u)_i=(\bar t)_i$. Note that if $\bar u\lesssim\bar t$, then $\bar u$ has at most the same number of infinite coordinates than that of $\bar t$. By \autoref{dim}, $D(\bar m)\lesssim\bar t$ is $d\mhyph\exists_1$ over $\mc L_{type}$.

     Now suppose $\bar s$ has $k$ infinite coordinates, and let $C_0,\ldots,C_N$ be all $\sim_r$-classes with fewer than $k$ infinite coordinates. Choose representatives $\bar n_i\in C_i$ with dimension $\bar t^i$, for all $i<N$. Now we claim the following formula expresses $D(\bar m)=\bar s$:
        \begin{itemize}
          \item $D(\bar m)\lesssim\bar s$; and
          \item ($\forall \bar n\not\indep\bar m)\ D(\bar n)\not\lesssim\bar t_i$, for all $i<N$.
        \end{itemize}
     This formula is $\forall_2$ over $\Ltype$ by \autoref{dim} and \autoref{indep}. Clearly, if $D(\bar m)=\bar s$ then the formula is satisfied. Conversely, if the formula is true of $\bar m$ then $D(\bar m)$ cannot be in any of the classes $C_i$ (by the second part), so it has at least $k$ infinite coordinates. But the first part shows $D(\bar m)$ has at most $k$ infinite coordinates, and all the finite ones agree with $\bar s$. Hence $D(\bar m)=\bar s$.
   \end{proof}

  Now we are finally ready to begin describing the Scott sentence.
  \begin{itemize}
    \item The first part $\Phi_M$ describes the dimension of $q_{\bar d}$ in $\mc M$, where $q\in Q$ and $\bar d$ is any tuple of $\mc M$ with type $\tp(d^0)$.

        The complexity of this part is $\Piin_2$ over $\Ltype$ by \autoref{dim}.
    \item The second part is a finite conjunction over $r\in R$ of sentences of the form $\theta^r(\bar s,\bar t)$, which says: There exists $\bar v,\bar w$ both with type $r$ that are dependent, such that $D(\bar v)=\bar s, D(\bar w)=\bar t$.

        The complexity of this part is $\exists_3$ over $\Ltype$ by \autoref{indep} and \autoref{Ddimension}.
    \item The third part is a finite conjunction over $r\in R$ of sentences of the form $\neg\theta^r(\bar s_i,\bar s_j)$.

        The complexity of this part is $\forall_3$ over $\Ltype$ by the same calculation as in the previous part.
    \item The last part is a finite conjunction over $r\in R$ of sentences expressing the cardinality of the set $r_{\bar s}(\mc M)$, which is a maximal independent set in $\mc M$ of tuples $\bar m$ realizing $r$ with $D(\bar m)=\bar s$.

        Saying $\bar m$ is a tuple realizing $r$ with $D(\bar m)=\bar s$ is $\forall_2$ over $\Ltype$, so expressing the dimension of $r_{\bar s}(\mc M)$ is $\Piin_4$ over $\Ltype$ (by a computation similar to \autoref{dim}).
  \end{itemize}
  Combining everything, we get a $\Siin_5$ over $\Ltype$ Scott sentence.
\end{proof}

We can further reduce the complexity with a bit of extra work.
\begin{cor}
  Every countable model of $T$ has a Scott sentence that is $\operatorname{d-\Siin_4}$ over $\Ltype$.
\end{cor}
\begin{proof}
  First, every such model has a Scott sentence that is $\Piin_5$ over $\Ltype$, by conjuncting over the negations of the $\Siin_5$ Scott sentences of all other (countably many) models of $T$. Now we are done by relativizing A. Miller's theorem \cite{AMiller}, which says any countable structure with both a $\Piin_{n+1}$ Scott sentence and a $\Siin_{n+1}$ Scott sentence has a $\operatorname{d-\Siin_n}$ Scott sentence, over $\Ltype$ (and an infinitary theory that says the type symbols in $\Ltype$ behave like types).
\end{proof}

We come close to showing the optimality of this result by showing that $\operatorname{d-\exists_4}$ cannot be replaced by $\operatorname{\exists_4}$.

\begin{thm}
  There exists an $\omega$-stable $T$ with countably many countable models and an $\mc M\sat T$, such that $\mc M$ does not have a Scott sentence that is $\Siin_4$ over $\Ltype$.
\end{thm}

\begin{proof}
  We take $T$ to be Shelah's theory with ENI-depth 2 but countably many models (see for example \cite[Chapter XVIII.4]{Baldwin}). We provide a relatively self-contained description of the theory here (following terminologies of \cite{Baldwin}), but those who are less detail-oriented should feel free to skip to the next paragraph. The (relational) language is $\set{E,R,G}\cup\set{P_i\mid i\in\omega}$, where each $P_i$ is unary, $E,R$ are binary, and $G$ is 5-ary. The first part of $T$ says $E$ is an equivalence relation, $R$ defines a directed acyclic graph on equivalence classes of $E$ (i.e. $R$ is $E$-invariant without loops), $G$ is symmetric in the second and third coordinates, and the $P_i$'s are pairwise disjoint and each intersects every $E$-class infinitely. Next, we say an element $a$ has \emph{order} $i$ if $P_i(a)$, or $\omega$ if it satisfies no $P_i$. The second part of the theory says for any $E$-classes $\alpha,\beta$ with $\alpha R \beta$, there exist unique $a\in\alpha, \set{b,c}\sbeq\beta$, such that $G(a,b,c,\cdot,\cdot)$ (equivalently $G(a,c,b,\cdot,\cdot)$) is nonempty; and whenever this is nonempty, it defines a \emph{proper link}, i.e. a bijection $\gamma_{\alpha,\beta}$ from $\alpha-\set a$ to $\beta-\set{b,c}$, interacting with order as follows: (1) If $a$ has infinite order then so do $b$ and $c$; (2) If $a$ has order $i<\omega$ then so does either $b$ or $c$; (3) If there is an $x$ with a different order than that of $\gamma_{\alpha,\beta}(x)$, then it is unique, with the order of $\gamma_{\alpha,\beta}(x)$ the same as that of $a$, and in this case the orders of $b,c$ needs to be different. The third and last part of the theory says for every $E$-class $\alpha$ and each $a\in\alpha$, there exist infinitely many $\beta$ properly linked to $\alpha$ via $a$ and some $\set{b,c}\sbeq\beta$; if $a$ has order $i<\omega$, then infinitely such $\beta$ exist with orders of $b,c$ equal to each other, and for each $d\in\alpha$ with order $j>i$ infinitely such $\beta$ exist with order of $b,c$ different and $\gamma_{\alpha,\beta}(d)$ having order $i$; Symmetrically, for every $E$-class $\beta$ and each $\set{b,c}\sbeq\beta$, there exist infinitely many $\alpha$ properly linked to $\beta$ via $\set{b,c}$ and some $a\in\alpha$; if $b$ has order $i<\omega$, then for each $c\in\beta$ with order at least $i$, and $e\in\beta$ with order $i$, neither equal to $b$, infinitely such $\alpha$ exist with orders of $a$ equal to $i$, and infinitely such $\beta$ exist with order of $a$ being $i$ and order of $\gamma\inv_{\alpha,\beta}(e)$ has order not equal to $i$. This concludes the description of the theory.

  The key takeaway of this theory is that each $\mc M\sat T$ consists of ``connected components'' (of $R$), where each ``node'' is actually an equivalence class (of $E$), and each class $\alpha$ is characterized by the dimension of a certain type $q_\alpha$ (which describes an element of order $\omega$). Additionally, for each component $C$, the set of dimensions of $q_\alpha$ over $\alpha\in C$ can only be $\set 0$, $\omega$, or $\set{\omega}$, which we name as Sort I, II, III components, respectively. The number of components of each sort determines a model of $T$ completely. Now, let $\mc M$ be the model where all three types of components appear infinitely often. It has a Scott sentence that is $\Piin_4$ over $\Ltype$: just say each type of component appears infinitely many often.

  Now we show $\mc M$ does not have a $\Siin_4$ over $\Ltype$ Scott sentence. To do this, we work in the (infinitary) theory where all finitary quantifiers are eliminated (by Morleyization) and all types are added to the language (requiring type predicates behave as expected), and then show $\mc M$ does not have a $\Siin_4$ Scott sentence in this language. The construction is as follows.

  Suppose towards a contradiction that $\mc M$ has a $\Siin_4$ Scott sentence in this language. Take an oracle $X$ computing the Scott sentence, and let $\set{\mc N_i}_{i\in\omega}$ be an $X$-computable enumeration of all $X$-computable structures in the same language, and let $\set{\mc R_i}_{i\in\omega}$ be an $X$-computable enumeration of all $X$-computable equivalence structures. Consider the set $A^{X}=\set{x\mid\mc R_x\text{ has infinitely many infinite classes}}$, which is $\Pi^0_4(X)$-complete (see \cite[Proof of Theorem 4.6]{bfrsm}). We will show $\set{i\mid\mc N_i\cong\mc M}$ is $\Pi^0_4(X)$-hard (among models of $T$), contradicting the existence of an $X$-computable $\Siin_4$ Scott sentence. We do this by reducing $A^X$ to the given set, as follows: Let $f$ be the computable function such that $\mc N_{f(i)}$ is the following structure: We put in infinitely many Sort I,II components; Then for each $y$, we put a new component $C_y$, generated by a class $\alpha_y$. Whenever we see the $y$-th component of $\mc R_x$ gain a new element, we put in a new realization of $q_{\alpha_y}$ (and change all other classes in $C_y$ correspondingly). It is clear that we always build a model of $T$, and $\mc N_{f(i)}\cong\mc M$ if and only if $i\in A^X$.
\end{proof}

\subsection{The use of parameters}
From now on, we will assume $T$ has uncountably many models. In the later subsections we will be using formulas with a fixed but omitted parameter extensively, which requires justification. Usually, such arguments for Vaught's conjecture (for example in \cite{SHM}) exhibits uncountably many models for $\Th(\mc M,a)$ (where $a$ is the parameter), but we need to keep track of the distinguishing formulas in addition. In this subsection, we show that omitting a fixed parameter does not cause more problems than adding an additional existential quantifier to our distinguishing formulas.

\begin{notn}
In this subsection only, we will abuse notation slightly and use $\mc L_1^{\text{in}}(T)$ to denote the set of $\Piin_{<\omega}$-\emph{formulas}, not just sentences, over $\Ltype(T)$.
\end{notn}

\begin{conj}
For a complete first order theory $T$ with uncountably many models, say the \emph{Type $\omega$-Vaught's conjecture with parameters} for $T$ holds if there exists some $n\in\omega$ such that there are $2^\ctb$ countable models $\set{\mc M_i}_{i\in I}$ and witnessing formulas $\set{\chi_i}_{i\in I}$, with each $\chi_i$ an $\Loin(T)$-formula, so that there is a way of assigning parameters $a_i$ to $\mc M_i$ that makes (for all $i,j\in I$) \[\mc M_i\sat\chi_j(\bar a_i)\iff j=i\]
and, additionally, the formulas $\chi_i$ has to be mutually exclusive, i.e. $T\sat\forall\bar x\neg(\chi_i(\bar x)\wedge\chi_j(\bar x))$ for all $i\ne j$.
\end{conj}
In other words, for each model in the list we can find one single formula that distinguishes it from every other one in the list, up to using a (possibly not first-order) way to choose the right parameters.

\begin{prop}
If $T$ has uncountably many models, then the Type $\omega$-Vaught's conjecture with parameters for $T$ implies the Type $\omega$-Vaught's conjecture for $T$.
\end{prop}
\begin{proof}
  Start by choosing $\set{\mc M_i}_{i\in I}, \set{\chi_i}_{i\in I}$ witnessing the Type $\omega$-Vaught's conjecture with parameters for $T$. We will choose a subset (with the same cardinality) of models that witnesses the Type $\omega$-Vaught's conjecture for $T$.

  Write $\mf c$ for $2^\ctb$ as ordinal. Inductively we build sequences $\set{\mc N_\alpha}_{\alpha<\mf c}$, $\set{\vp_\alpha}_{\alpha<\mf c}$ witnessing the Type $\omega$-Vaught's conjecture for $T$, where each $\mc N_\alpha$ is some $\mc M_i$, and every pair $\mc N_\alpha$ and $\mc N_\beta$ (say $\beta<\alpha$) are distinguished by $\exists\bar x\vp_\alpha(\bar x)$. To this end, in the construction we will also need to maintain the following induction hypothesis: for all $\beta<\alpha$ and any tuple $\bar a$ from $\mc N_\beta$, $\mc N_\beta\not\sat\vp_\alpha(\bar a)$; and $\mc N_\alpha\sat\exists\bar x\vp_\alpha(\bar x)$.

  We induct on $\alpha$, starting with two empty sequences (which clearly satisfies the induction hypothesis). Now given the induction hypothesis holds for all $\beta<\alpha$, we choose $\mc N_\alpha, \vp_\alpha$ in the following way. There are only at most $|\alpha|+\ctb$ ways to choose a pair $(\mc N_\beta,\bar a)$, where $\bar a$ is a tuple from $\mc N_\beta$. Since the $\chi_i$'s are mutually exclusive, there are at most $|\alpha|+\ctb$ many $i$'s so that $\mc N_\beta\sat\chi_i(\bar a)$. As $\mf c>|\alpha|+\ctb$, we can choose some $i\in I$ that does not satisfy the previous condition. Now take $\mc N_\alpha=\mc M_i$ and $\vp_\alpha=\chi_i$. The induction hypothesis is maintained due to the way we chose $i$. Now any pair $\mc N_\alpha, \mc N_\beta$ are distinguished by the sentence $\exists\bar x\vp_{\max(\alpha,\beta)}(\bar x)$.
\end{proof}

\begin{rmk}\label{paramConvention}
Due to the above proposition, in the following subsections we will be freely using formulas with parameters as distinguishing formulas, and it will be easy to check that the formulas involved satisfy the assumption of this proposition.

Also as a result, we are adding an existential quantifier to the distinguishing formulas. Thus in the statement of \autoref{main}, the models are pairwise $\Siin_6(\Ltype)$-inequivalent even though the analysis below will use $\Piin_5(\Ltype)$ formulas only.
\end{rmk}

\subsection{Shelah-Harrington-Makkai's analysis}
Shelah-Harrington-Makkai's analysis gives necessary and sufficient conditions for $T$ to have countably many countable models as follows. Since we are assuming $T$ has uncountably many countable models, one of the four conditions below must fail.

\begin{thm}[{\cite[Theorem 3.1]{SHM}}]\label{ctbconds}
For any complete $\omega$-stable theory $T$, it has countably many countable models if and only if all of the following conditions hold:
\begin{enumerate}[label=(\arabic*), ref=\thethm(\arabic*)]
  \item\label{ctbconds:fincls} (\cite[Proposition 2.9]{SHM}) There are only finitely many equivalence classes of supportive types under the equivalence relation $\not\perp$.
  \item\label{ctbconds:fineni} (\cite[Proposition 2.10]{SHM}) There are only finitely many equivalence classes of ENI types $\not\perp$  to any fixed finite set $A$ under the equivalence relation $\not\perp$.
  \item\label{ctbconds:depth} (\cite[Proposition 2.4]{SHM}) For any ENI type $p$ and any $C/A$ stationary with weight 1, $p\not\perp C, p\perp A$, we have $C/A\not\perp\emptyset$.
  \item\label{ctbconds:ndop} (\cite[Proposition 2.3]{SHM}) $T$ has ENI-NDOP.
\end{enumerate}
\end{thm}
\renewcommand{\itemautorefname}{Theorem}

In each of these cases, Shelah-Harrington-Makkai proves that the failure of each condition witnesses $2^\ctb$ models. We will analyze such constructions and show that those models are pairwise distinguishable in $\Loin(T)$.

Now we start analyzing the four conditions in \autoref{ctbconds} case-by-case (roughly in increasing order of difficulty). Some notational conventions: We may start using lowercase letters to denote tuples, without adding bars. And superscripts on types are for indexing only (to avoid double subscripts).

\begin{rmk}
We would like to remind the reader that while Shelah-Harrington-Makkai's analysis in every case below is relative to a parameter, we will not be mentioning them explicitly except when necessary, in light of \autoref{paramConvention}.
\end{rmk}

\subsection{Infinite width on the first level of the ENI tree}
Suppose \autoref{ctbconds:fincls} fails. Following the analysis in Shelah-Harrington-Makkai, we can find:
\begin{itemize}
  \item A sequence $\set{q_n}_{n\in\omega}$ of supportive types such that any isomorphic copy of $q_m$ is orthogonal to any isomorphic copy of $q_n$ for all $m\ne n$;
  \item Tuples $b_n=\dom q_n$;
  \item Some realizations $c_n$ of $q_n$; (We may assume $c_n\speq b_n$ by throwing in $b_n$ and extending $q_n$.)
  \item ENI types $p^n$ needing $c_n/b_n$; (We may assume $p^n$ is non-isolated.)
  \item Tuples $d_n=\dom p^n$; (We may similarly assume $d_n\speq c_n$.)
\end{itemize}
Using the above information, a family $\mc M_X$ of models of $T$ for all $X\sbeq\omega$ is constructed, where $X$ can be recovered from $\mc M_X$ as follows: For any $n\in\omega$, $n\in X$ if and only if $\mc M_X$ satisfies:
\[\exists b\exists c\exists d(\tp(bcd)=\tp(b_nc_nd_n)\wedge\dim(p^n_d,\mc M_X)=\omega).\]
In our notation, this is equivalent to the formula $\exists b\exists c\exists d(t(bcd)\wedge\psi_{r,p^n,\omega}(d))$, where $t$ is the $\mc L_{type}$-predicate for $\tp(b_nc_nd_n)$ and $r=\tp(d_n)$. In other words, the distinguishing formula $\chi_X$ for $\mc M_X$ is \[(\bigdoublewedge_{n\in X}\exists b\exists c\exists d(t(bcd)\wedge\psi_{r,p^n,\omega}(d)))\wedge(\bigdoublewedge_{n\notin X}\neg\exists b\exists c\exists d(t(bcd)\wedge\psi_{r,p^n,\omega}(d))).\] This is $\Piin_4$ over $\Ltype(T)$, so we are done.

\subsection{Infinite width on the second level of the ENI tree}
Suppose \autoref{ctbconds:fineni} fails. Take a finite set $a$ witnessing the failure. Following the analysis in Shelah-Harrington-Makkai, we can find:

\begin{itemize}
  \item A sequence $\set{p^n}_{n\in\omega}$ of ENI types whose domain is in the prime model over $a$, such that any isomorphic copy of $p^m$ is orthogonal to any isomorphic copy of $p^n$ for all $m\ne n$;
  \item Tuples $b_n=\dom p^n$.
  \item Types $q_n=\tp(b_n)$.
\end{itemize}

Using this information, a family $\mc M_X$ of models of $T$ for all $X\sbeq\omega$ is constructed, where $X$ can be recovered from $\mc M_X$ as follows: For any $n\in\omega$, $n\in X$ if and only if $\mc M_X$ satisfies:
\[\exists b(\tp(b)=q_n\wedge\dim(p^n_b,\mc M_X)<\omega).\]
In our notation, this is equivalent to the formula $\exists b\,\neg\psi_{q_n,p^n,\omega}(b)$. This is $\Siin_2$ over $\Ltype(T)$. Hence, we are able to distinguish $2^\ctb$ many models with formulas that are $\Piin_3$ over $\Ltype(T)$, so we are done.

\subsection{ENI-depth more than 2}
Suppose \autoref{ctbconds:depth} fails. Following the analysis in Shelah-Harrington-Makkai, we can find finite tuples $a\sbeq b\sbeq c$, ENI type $p$ with domain $c$, such that $c/b,b/a$ stationary of weight 1, $p\perp b, c/b\perp a$. We will introduce some auxiliary notions used in the proof, verifying their definability (in $\Loin(T)$) at the same time. In the end, this will lead to a definable way of distinguishing $2^\ctb$ many models.

\begin{itemize}
  \item A \emph{class} is an equivalence class of an element $x$ which has the same type as $b$, under the equivalence relation $\not\indep$. The class of any $x$ is denoted $[x]$. Clearly, the statement $y\in[x]$ is quantifier-free definable in $\Ltype$, because of \autoref{indep} and the fact that we know the type of $x$.
  \item For an element $x$ having the same type as $b$, $m(x)$ is the supremum of all $m\le\omega$ such that there is a system $\set{y_i}_{i<m}$ independent over $x$ such that $\tp(xy_i)=\tp(bc)$ and each $p_{y_i}$ has finite dimension in the model. Saying $m(x)\ge n$, for any fixed $n\in\omega$, is $\Siin_2(\Ltype)$-definable using \autoref{indep} and \autoref{dim}.
  \item For a class $\mc B$, $m(\mc B)$ is the supremum of $m(x)$ over all $x$ in the class. Combining the previous arguments, for a fixed $n\in\omega$, $m([x])\ge n$ is also $\Siin_2(\Ltype)$-definable from $x$.
\end{itemize}
At last, a family of models $\mc M_X$ (for $X\sbeq\omega\bsl\set0$) is defined in the proof, where $n\in X\iff\mc M_X$ satisfies that: There is a class $\mc B$ with $m([\mc B])=n$; or equivalently, there is an element $x$ such that $m([x])=n$. The statement that the above correspondence holds for all $n\in\omega$ is $\Piin_4(\Ltype)$ and distinguishes $\mc M_X$ from other models in the family, so we are done.

\subsection{ENI-DOP}
The last condition is significantly more difficult than the other ones: Suppose \autoref{ctbconds:ndop} fails. We refer to the analysis by Baldwin in the proof of \cite[Theorem XVI.3.1]{Baldwin}.

There, a family $\mc M_G$ of models are defined for connected symmetric graphs $G$ on a countable set of vertices (of which there are $2^\ctb$ many). But instead of coding $G$ (with domain) directly, another graph $H$ (as well as domains of $G$ and $H$) can be recovered from $\mc M_G$. $H$ is defined in a way such that $V(G)\sbeq V(H)$, and that $b,b'\in V(G)$ are connected by an edge in $G$ if and only if there exists an infinite set $A_{bb'}\sbeq V(H)\bsl V(G)$ such that every $a\in A_{bb'}$ has an edge with both $b$ and $b'$ in $H$. (For different $(b,b')$, the sets $A_{bb'}$ are disjoint.) 

Following the notation there, we write $V(G)=B, V(H)\bsl V(G)=A$. (Technically $B$ has one more point than $V(G)$ which is connected to every vertex in $V(G)$, but for our purpose it's fine to throw it in $B$ and slightly restrict the class of possible $G$'s.)

For our purpose, the witnessing models will be $\set{\mc N_X}_{X\sbeq\omega\bsl\set{0,1,2}}$, where $\mc N_X$ is $\mc M_G$ for some (fixed) $G$ that contains an $n$-cycle iff $n\in X$ (for all $n\in\omega, n\ge 3$). The distinguishing formulas say that $\mc N_X$ contains an $n$-cycle iff $n\in X$. It should be clear that as long as the domain and edge relation of $G$ is $\Loin$-definable from $\mc M_G$, then we are done. This is what we will check below.

Before we discuss the graph recovery process, here is some information that we can gather from the fact that ENI-DOP (and thus DOP) holds:
\begin{itemize}
  \item Finite tuples $n_0,n_1,n_2,n_3$ such that $n_0\sbeq n_i\sbeq n_3$ for $i=1,2$, $n_1\indep_{n_0}n_2, n_1\rhd_{n_2}n_3,n_2\rhd_{n_1}n_3, n_1\cup n_2\rhd_{n_0}n_3$. (The definition of $\rhd$ is not relevant to our analysis.)

  For notational convenience, we will assume $n_0$ is the empty set by incorporating (the types) $n_0$ in all types that we will mention. So we are always ``working over $n_0$'' from now on.
  \item Types $q_i=\tp(n_i)$ for $i=1,2$.
  \item Type $p$ over $n_3$ such that $p\dashv n_i$ for $i=1,2$.
\end{itemize}
Another piece of notation: Let $Q(\mc M)$ denote the set of elements realizing of $q_1$ or $q_2$ in $\mc M$. This is clearly $\Ltype$-definable in $\mc M$ (by $q_1\vee q_2$). In the construction, elements of $A$ will realize $q_1$ and elements of $B$ will realize $q_2$, so $A\cup B\sbeq Q$.

Yet another relevant definition: say $\ang{a,b,e_{ab}}$ is a \emph{germane triple} if $abe_{ab}\sat\tp(n_1n_2n_3)$, and let $p_{ab}$ denote a copy of $p$ over $e_{ab}$ for some germane triple containing $a,b$. The $\Ltype$-definability of $p_{ab}$ from $a,b$ is similarly clear (provided we do not care about the actual $e_{ab}$, which is the case here).

Now we verify that the graph $G$ (including its domain) can be recovered from $M_G$. Recall from the last subsection that if $a$ is any element (having a fixed type), then $[a]$, the $\not\indep$-equivalence class of $a$, is quantifier-free definable in $\Ltype$.

\begin{itemize}
  \item We first recover the domains $A$ and $B$. There are parameters $\hat a,\hat b$ such that for all $c\in Q, [c]\in[A]\iff\dim(p_{c\hat b},\mc M_G)<\omega$; and $[c]\in[B]\iff\dim(p_{\hat ac},\mc M_G)<\omega$. Clearly these domains are $\Siin_2(\Ltype)$-definable using \autoref{dim} (up to modifying $p_{\bar m}$ with the germane triple definition).
  \item Now we recover the edge relation. For any infinite $X\sbeq A, b\in B$, say $X\mc Fb$ if $\ang{x,b}\in H$ for cofinitely many $x\in X$. Then we have:\footnote{Technically, we need the germane pairs to follow some normal position convention, as in \cite[Convention 3.14]{Baldwin}. But as the proposition preceding that convention shows, switching the order of $a$ and $b$ everywhere, if necessary, suffices.}
      \[X\mc Fb\iff\exists b'\in[b]\forall^\infty x\in X R(x,b').\] where $\forall^\infty x\in X$ means ``for cofinitely many $x\in X$,'' and $R(x,b)$ means $\exists x'\in[x]\dim(p_{x'b},\mc M_G)<\omega$ (thus $\Siin_2(\Ltype)$-definable). Using this, we can recover the edge relation of $G$ without using set quantification:
      \begin{prop}
      For $b_1,b_2\in B$, we have \[\ang{b_1,b_2}\in G\iff{\exists b_1'\in[b_1]}{\exists b_2'\in[b_2]}{\exists^\infty x}(R(x,b_1')\wedge R(x,b_2')).\]
      In particular, the edge relation on $G$ can be recovered from $\mc M_G$ using a $\Siin_4(\Ltype)$ formula.
      \end{prop}
      \begin{proof}
        ($\Leftarrow$) Choose witnesses $b_1',b_2'$, and let $X$ be the infinite set of witnesses $x$. Then we have $X\mc Fb_1$ and $X\mc Fb_2$, so $b_1,b_2$ share infinitely many common neighbors in $H$, thus are connected in $G$ (by definition of $H$).

        ($\Rightarrow$) By definition of $H$, let the infinite set of common $H$-neighbors of $b_1,b_2$ be $X_0$. Since $X_0\mc Fb_1, X_0\mc Fb_2$, choose corresponding witnesses $b_1'\in[b_1], b_2'\in[b_2]$. Let $F_1,F_2$ be finite sets such that $R(x,b_i')$ for all $x\in X_0\bsl F_i$. Let $X=X_0\bsl(F_1\cup F_2)$, clearly also infinite. Then $b_1',b_2',X$ witness the satisfaction of the RHS formula.
      \end{proof}
      The above proposition presents $G$ as a quotient structure (by the relation $\not\indep$), such that:
      \begin{itemize}
        \item the quotient relation is quantifier-free definable in $\Ltype$;
        \item the domain is $\Siin_2(\Ltype)$-definable; and
        \item the edge relation is $\Siin_4(\Ltype)$-definable.
      \end{itemize}
      As a result, any positive atomic fact about $G$ can be recovered from $\mc M_G$ in a $\Siin_4(\Ltype)$-definable way. Now given any $X\sbeq\omega$, consider the sentence that says for all $n\in\omega$, the copy of $G$ recovered from $\mc M_G$ has an $(n+3)$-cycle iff $n\in X$. Such statements are $\Piin_5(\Ltype)$ and distinguishes $2^\ctb$ many $\mc M_G$, so we are done.
\end{itemize}

With all of the above, we have finally established our main result.

\main

In the uncountable case, \autoref{paramConvention} explains why $\Siin_6(\Ltype)$ is needed rather than $\Piin_5(\Ltype)$, and then we get the stated $\Piin_6(\Ltype)$ result by taking negations.

\section{Open Questions}
First, we would like to ask whether the uniform finite bounds given in \autoref{main} are optimal. Namely:
\begin{quest}
  \begin{itemize}
    \item Is there a complete $\omega$-stable $T$ with countably many countable models, such that some model has no $\Piin_4(\Ltype)$ Scott sentence?
    \item Is there a complete $\omega$-stable $T$ with uncountably many countable models, such that there are only countably many equivalence classes of such models up to $\Piin_5(\Ltype)$-equivalence?
  \end{itemize}
\end{quest}

In addition, as mentioned in \autoref{wVCcomparison}, the type $\omega$-Vaught's conjecture implies only the $(\omega+\omega)$-Vaught's conjecture. While we can reduce the complexity down to $\omega+6$ as in \autoref{wVCbound}, it still remains open as to whether the $\omega$-Vaught's conjecture for $\omega$-stable theories hold. The main obstacle is:
\begin{quest}
Is there a boundedly axiomatizable $\omega$-stable theory with a type that is not boundedly axiomatizable?
\end{quest}
A negative answer to this question would establish the $\omega$-Vaught's conjecture for $\omega$-stable theories. On the other hand, the question has a positive answer without the $\omega$-stable requirement \cite{bddUnbdd}, and the example there is strictly superstable. Though even if an $\omega$-stable example were to exist, it would still leave open the $\omega$-Vaught's conjecture for $\omega$-stable theories.

\begin{quest}
  Does the $\omega$-Vaught's conjecture hold for $\omega$-stable theories?
\end{quest}

\section*{AI Disclosure}

The author used ChatGPT-5.6 Sol for proofreading the paper, in particular checking for grammatical and mathematical errors.
The research was completed without the use of AI tools, and the author takes full responsibility for the content of this paper.

\printbibliography

\end{document}